\documentclass[10pt, a4paper]{amsart}
\usepackage{geometry} 
\usepackage{color}
\usepackage{graphicx}
\usepackage{mathrsfs}
\usepackage{bm}
\newtheorem{theorem}{Theorem}[section]
\newtheorem{corollary}[theorem]{Corollary}
\newtheorem{definition}{Definition}[section]
 
\newtheorem{proposition}[theorem]{Proposition}
\newtheorem{claim}[theorem]{Claim}
\newtheorem{example}[theorem]{Example}
\newcommand{\supp}{\operatorname{supp}}

\def\bS{\mathbf{S}}
\def\cU{\mathcal{U}}
\def\bU{\mathbf{U}}
\def\bW{\mathbf{W}}
\def\bu{\mathbf{u}}
\def\bw{\mathbf{w}}
\def\bX{\mathbf{X}}

\def\bx{\mathbf{x}}

\def\bd{\mathbf{d}}

\def\bbU{\mathbb{U}}
\def\bbW{\mathbb{W}}
\def\bV{\mathbf{V}}
\def\bv{\mathbf{v}}
\def\bt{\mathbf{t}}
\def\br{\mathbf{s}}
\def\br{\mathbf{r}}

\def\R{\mathbb{R}}
\def\N{\mathbb{N}}
\def\bM{\mathbb{M}}

\def\Mdos{\mathbb{M}^{2\times 2}}
\def\F{\mathbf F}
\def\bbH{\mathbb H}
\def\bT{\mathbf T}

\def\bbP{\mathbb P}
\def\A{\mathcal A}
\def\bA{\mathbf A}
\def\C{\mathcal C}
\def\cM{\mathbf{M}^{m\times N}}
\def\bcero{\mathbf 0}
\def\skeleton{\mathscr{S}_l}
\def\skeletonj{\mathscr{S}_{l, j}}
\def\skeletondos{\mathscr{S}_2}
\def\clase{\mathscr{C}}
\def\pele{\mathscr{P}}
\def\bbu{\bm{u}}
\def\bbv{\bm{v}}
\def\bbx{\bm{x}}
\def\bby{\bm{y}}

\numberwithin{equation}{section}

\title[Rank-one convexity and quasi-convexity]{Morrey's conjecture: rank-one convexity implies quasi-convexity for two-dimensional, two-component maps}
\author{Pablo Pedregal}
\thanks{Departamento de Matemáticas, Universidad de Castilla-La Mancha, 13071 Ciudad Real, SPAIN. Supported by grants PID2023-151823NB-I00, and  SBPLY/23/180225/000023}

\begin{document}
\maketitle

    \begin{abstract}
We prove that for two-component maps in dimension two, rank-one convexity is equivalent to quasiconvexity. The essential tool for the proof is a fixed-point argument for a suitable set-valued map going from one component to the other that preserves decomposition directions within the $(H_n)$-condition formalism. The existence of a fixed point ensures that, in addition to keeping decomposition directions, joint volume fractions are respected as well, leading to the fundamental fact that every two-dimensional, two-component gradient can be reached by lamination. When maps have more than two components, fixed points exist for every combination of two components, but they do not match in general. Higher dimension would require further insight on how to organize and deal with triangulations for piece-wise affine maps. 
    \end{abstract}


\vspace{10pt} {\bf AMS(MOS) subject classifications.} 65L05,
65L20, 47J30, 65D05.


\section{Introduction}
One of the main ingredients of the direct method of the Calculus of Variations (\cite{DacorognaH}) to show existence of minimizers for an integral functional of the kind
$$
I(\bbu)=\int_\Omega\psi(\nabla \bbu(\bbx))\,d\bbx
$$
is its weak lower semicontinuity. Here $\Omega\subset\R^N$ is a regular (Lipschitz), bounded domain, and feasible mappings $\bbu:\Omega\to\R^m$ are smooth or Lipschitz, so that $\nabla \bbu$ is a $m\times N$-matrix at each point $\bbx\in\Omega$. The weak lower semicontinuity property is in turn equivalent to suitable convexity properties of the continuous integrand $\psi:\cM\to\R$. Morrey (\cite{Morrey}, \cite{MorreyB}) proved that this weak lower semicontinuity (in $W^{1, \infty}(\Omega; \R^m)$) is equivalent to the quasi-convexity of the integrand $\psi$, namely, 
$$
\psi(\F)\le\frac1{|D|}\int_D\psi(\F+\nabla \bbv(\bbx))\,d\bbx
$$
for every $\F\in\cM$, and every test map $\bbv$ in $D$. This concept does not depend on the domain $D$, and can, equivalently, be formulated in terms of periodic mappings (\cite{Sverak}) so that such a density $\psi$ is quasiconvex when
$$
\psi(\F)\le\int_Q\psi(\F+\nabla \bbv(\bby))\,d\bby
$$
for all $\F\in\cM$, and every periodic mapping $\bbv:Q\to\R^m$. Here $Q\subset\R^N$ is the unit cube. 

Unfortunately, the issue is far from settled by simply saying this, since even Morrey realized that it is not at all easy to decide when a given density $\psi$ enjoys this property. For the scalar case, when either of the two dimensions $N$ or $m$ is unity, quasi-convexity reduces to usual convexity. But for genuine vector situations, it is not so. As a matter of fact, necessary and sufficient conditions for quasi-convexity in the vector case ($N, m>1$) were immediately sought, and important new convexity conditions were introduced:
\begin{itemize}
\item Rank-one convexity. A continuous integrand $\psi:\cM\to\R$ is said to be rank-one convex if
$$
\psi(t_1\F_1+t_2\F_2)\le t_1\psi(\F_1)+t_2\psi(\F_2),\quad t_1+t_2=1, t_1, t_2\ge0,
$$
whenever the difference $\F_1-\F_2$ is a rank-one matrix. 
\item Poly-convexity. Such an integrand $\psi$ is poly-convex if it can be rewritten in the form 
$$
\psi(\F)=g(\bM(\F))
$$ 
where $\bM(\F)$ is the vector of all minors of $\F$, and $g$ is a convex (in the usual sense) function of all its arguments.
\end{itemize}

It was very soon recognized that quasi-convexity implies rank-one convexity (by using a special class of test fields), and that poly-convexity is a sufficient condition for quasi-convexity (\cite{Ball}) The task suggested itself as trying to prove or disprove the equivalence of these various kinds of convexity. In the scalar case all three coincide with usual convexity, so that we are facing a purely vector phenomenon. It turns out that these three notions of convexity are different, and counterexamples of various sorts have been found over the years. See \cite{AlibertDacorogna}, \cite{DacorognaDouchetGangboRappaz}, \cite{SerreA}, \cite{Terpstra}.

If we focus on the equivalence of rank-one convexity and quasi-convexity, Morrey conjectured that they are not equivalent (\cite{Morrey}), though later he simply stated it as an unsolved problem (\cite{MorreyB}). The issue remained undecided until the surprising counter-examble by V. Sverak (\cite{Sverak}) after some other additional and very interesting results (\cite{SverakB}, \cite{SverakC}, \cite{SverakE}). What is quite remarkable is that the original counter-example is only valid when $m\ge3$, and later attempts to extend it for $m=2$ failed (\cite{bandeiraornelas}, \cite{PedregalH}, \cite{PedregalSverakB}). See also \cite{grabovsky} for more such examples from a different viewpoint again in cases where $m>2$. Other counterexamples have not been found. Some efforts by the author were definitely discarded in \cite{sebsze}. References \cite{faracolaszlo}, and \cite{KristensenB} are also relevant here.

The situation for two-component maps has, therefore, stayed unsolved, though some evidence in favor of the equivalence has been gathered throughout the years. In particular, when additional ingredients or properties are assumed, the equivalence can some times be shown. There is a bunch of very interesting works in this regard; see \cite{benkru}, \cite{ChaudMuller}, \cite{ghiba}, \cite{grabovskyd}, \cite{kruzik}, \cite{martin}, \cite{Muller}, \cite{MullerB}, \cite{ParryA}, \cite{vossd}. It is also interesting to point out that for quadratic densities, rank-one convexity and quasi-convexity are equivalent regardless of dimensions. This has been known for a long time (\cite{Ball}, \cite{MorreyB}), and it is not difficult to prove it by using Plancherel's formula. A different point of view is taken in \cite{BandeiraPedregal}.  Another field where the resolution of this equivalence for two components maps would have an important impact is the theory of quasiconformal maps in the plane. There is a large number of references for this topic. See \cite{astalaiwaniecmartin} for a rather recent account, or \cite{astala} for a more focused article. In particular, if the equivalence between rank-one convexity and quasi-convexity for two component maps turns out to be true, then the norm  of the corresponding Beurling-Ahlfors transform equals $p^*-1$ (\cite{iwaniec}). More consequences would likely follow if such equivalence is proved. 

Another chapter where many recent efforts have been made, given the intrinsic difficulties of analytical ideas, is concerned with the numerical evidence for a potential counter-example in favor of Morrey's conjecture. None of them has turned out to be conclusive. See \cite{awi}, \cite{dong}, \cite{guerra}, \cite{voss}. 

In this note we prove that indeed for $m=N=2$, rank-one convexity is equivalent to quasi-convexity. The way in which we are going to think about the problem is by using the dual formulation of this equivalence through Jensen's inequality. What we will actually show is that, when $m=N=2$, every homogeneous gradient Young measure is a laminate. See Chapter 9 in \cite{PedregalI}, and \cite{pedlam}. Equivalently, we will focus on showing that every periodic gradient can be achieved by lamination. 

More specifically,  take $l\in\N$, and $\tau_l$, a regular triangulation of the unit cube $Q\subset\R^2$ with the three normals $(1, 0)$, $(0, 1)$, and $(1, 1)$. $l$ is a parameter indicating a certain level of discretization.
As $l$ becomes larger and larger, elements in $\tau_l$ are finer and finer triangles with the same normals. Suppose we are given a gradient $(\nabla u, \nabla v)$ with two components 
$$
(u, v):Q\subset\R^2\to\R^2,
$$
where $Q$ is the unit cube in $\R^2$, which is $Q$-periodic, continuous, piecewise-affine with respect to triangulation $\tau_l$. By a standard density argument about approximation by continuous, piece-wise affine mappings, it suffices, to reach our goal, to show that the corresponding discrete, homogeneous underlying gradient Young measure is a laminate.

\begin{theorem}\label{objetivo}
\begin{enumerate}
\item Let $l\in\N$ be arbitrary. For every pair $(u, v)$ of $Q$-periodic, continuous, $\tau_l$-piece-wise affine functions, the discrete probability measure 
$$
\nu_{(\nabla u, \nabla v)},\quad \supp(\nu)\subset\R^{2\times2},
$$ 
associated with its gradient is a laminate.
\item Rank-one convexity implies quasiconvexity for two-dimensional, two-component maps.
\end{enumerate}
\end{theorem}

What is essential or special about $m=2$? This is  a question that one has to understand, as it seems quite central for a final resolution of the problem. The answer turns out to be quite enlightening: for two component maps, one can define an appropriate map going from one component to the other, and show the existence of a fixed point for such a map that translates into a rank-one decomposition for any such two-component gradient. For more than two components, more than one map would be involved, and fixed points for every couple of components may not match. This fixed point result (Kakutani's) is classical and nothing but a natural generalization of the usual Brower fixed point theorem. 

In a more explicit way, the two-component map $(u, v)$ establishes, through triangulation $\tau_l$, a very clear way of moving from manipulations on the gradient of the first component $\nabla u$ to the same manipulations on the gradient $\nabla v$ of the second component by simply replacing $\bu_i$ by the corresponding $\bv_i$ in the same element of the triangulation $\tau_l$, if the finite support of $(\nabla u, \nabla v)$ is the set of pairs $\{(\bu_i, \bv_i)\}_i$. 
This procedure is incorporated in the definition of our mapping. Such map is in charge of keeping track of decomposition directions as in the definition of laminates and $(H_n)$-conditions (\cite{DacorognaE}, \cite{pedlam}). We assume readers to be familiar with this material as it is essential to understand our perspective. 

Given a probability measure supported in the discrete set of vectors $\{\bu_i\}_i$ of the first gradient $\nabla u$, that is decomposed in the form of a $(H_n)$-condition along a set of successive directions and weights, we focus on those decompositions, performed in the same way for the second gradient $\nabla v$, that preserve the family of decomposition directions coming from the first component, but possibly for a different family of volume fractions or weights. If we define a set-valued operation going from the first vector of weights for the first component to the set of weights of the second component preserving decomposition directions at all levels, intuitively a fixed-point for such a map would respect:
\begin{enumerate}
\item decomposition directions for both components (this is ensured by the definition of the map itself); and 
\item equal volume fractions for the two components jointly, because the passage from one component to the other through the operation $\bu_i\mapsto\bv_i$  respects such volume fractions for a fixed point.
\end{enumerate}
Therefore fixed points for such a map are identified with joint, i.e. simultaneously in the two components, $(H_n)$-conditions
whose decomposition directions are parallel, i.e. with laminates. Our claim, then, reduces to proving the existence of at least one fixed point for such a map. 

Most of the technicalities are related to showing that a suitable framework can be set up so that the appropriate assumptions hold for the fixed-point result to be applied.
One crucial issue, though, is to understand what is special about a probability measure associated with a gradient $(\nabla u, \nabla v)$, since we know that not every probability measure supported in $\R^{2\times2}$ should allow the treatment through such fixed point argument. 

It will soon be understood why we do not deal with the case $N>2$. The strategy of the proof hinges on various basic facts whose extension to higher dimension is not clear or does not hold. On the one hand, we need to manipulate in a rather explicit manner periodic, continuous, piece-wise linear functions with respect to triangulations of the unit cube in $\R^N$. As soon as $N=3$, this becomes highly technical where enumerative procedures for planar interfaces between elements turn out to be quite tricky to organize as in the practical use of finite elements. Beyond this difficulty, Claim \ref{corte} below can hardly admit a parallel result in dimension $N=3$ or higher. Further insight is therefore required to treat the two-component, higher dimensional case, either to find suitable substitutes for those facts or for finding a counter-example. 

Before moving to discuss our main proof, we would like to mention that one of the main applied fields where vector variational problems are relevant is non-linear elasticity (\cite{Ball}). In particular, polyconvexity has played a major role in existence results. See also \cite{Ciarlet}. A main hypothesis to be assumed in this area is the rotationally invariance, as well as the behavior for large deformations. See \cite{DacorognaKoshigoe} for a discussion on all these notions of convexity under this invariance. Higher-order theories have also been explored, at least from an abstract point of view (\cite{Dal Maso et al.}, \cite{Meyers}). More general concepts of quasiconvexity have been introduced in \cite{Fonseca-Muller}. Recent interesting results about approximation by polynomials are worth mentioning \cite{Heinz}. 
Explicit examples of rank-one convex functions can be found in various works: \cite{BandeiraPedregal}, \cite{DacorognaDouchetGangboRappaz}, \cite{SverakB}, among others. See also \cite{VanHove1}, \cite{VanHove2}. The recent book \cite{rindler} is to be considered.

Finally, it may be instructive to write Theorem \ref{objetivo} in a more transparent form for readers not familiar with the subleties of vector variational problems. 
\begin{theorem}
Let 
$$
\psi(\bX):\Mdos\to\R
$$
be an arbitrary function. The two following statements about $\psi$ are equivalent:
\begin{enumerate}
\item $\psi$ is rank-one convex, i.e. 
$$
\Psi(t\F_1+(1-t)\F_0)\le t\psi(\F_1)+(1-t)\psi(\F_0)
$$
whenever $\F_1-\F_0$ is a matrix of rank-one;
\item $\psi$ is quasi-convex, i.e. for every bounded regular domain $D\subset\R^2$ with $|\partial D|=0$, for every matrix $\F\in\Mdos$, and for every Lipschitz map $\bbu(\bbx):D\to\R^2$ belonging to $W^{1, \infty}_0(D; \R^2)$ vanishing at $\partial D$, we have
$$
|D|\psi(\F)\le\int_D\psi(\F+\nabla\bbu(\bbx))\,d\bbx.
$$
\end{enumerate}
\end{theorem}
This result is false in the case $\cM$ for $m\ge3$ and every $N\ge2$ (\cite{Sverak}), and it is not known if it is true for $m=2$ and $N\ge3$.

\section{The overall strategy}\label{diez}
It may be instructive to describe in heuristic but explicit terms what we are trying to accomplish and how we plan to tackle the proof of Theorem \ref{objetivo}, so that readers may know in advance where we are heading before getting into precise statements. 

Suppose we are given a discrete probability measure $\mu$ with finite support in $\Mdos$, the space of $2\times2$-matrices, and a vanishing barycenter
$$
\mu=\sum_it_i\delta_{\F_i},\quad t_i>0, \sum_it_i=1,\quad \sum_it_i\F_i=\bcero. 
$$
How is one to show that $\mu$ can indeed be reached by lamination? One way or another, we need to find a suitable representation of $\mu$ in the form
\begin{equation}\label{decomposition}
\mu=\sum_j s_j\delta_{\bA_j},\quad s_j>0, \sum_js_j=1, \bA_j\in\{\F_i\}, \quad t_i=\sum_{j: \bA_j=\F_i}s_j,
\end{equation}
in such a way that with the initial set of pairs
\begin{equation}\label{decompositiondos}
(s_1, \bA_1),\quad (s_2, \bA_2), \quad\dots\quad (s_m, \bA_m), 
\end{equation}
we can build a successive collection of $(H_n)$-conditions, according to its proper rule, and proceeding recursively up until the final pair $(1, \bcero)$. That basic rule consists in selecting two pairs $p$, $q$, from \eqref{decompositiondos} in such a way that $\bA_p-\bA_q$ is a rank-one matrix, and passing from \eqref{decompositiondos} to the new collection of pairs replacing those two selected pairs by its associated mixture
$$
(s_p+s_q, \bA_{pq}),\quad \bA_{pq}=\frac{s_p}{s_p+s_q}\bA_p+\frac{s_q}{s_p+s_q}\bA_q,
$$
and retaining unchanged all the others. In performing such step or mixture we refer to the difference $\bA_p-\bA_q$ as the corresponding decomposition direction, and the weights
$$
\lambda_p=\frac{s_p}{s_p+s_q}, \quad\lambda_q=\frac{s_q}{s_p+s_q}, \quad \lambda_p+\lambda_q=1,
$$
as the associated relative volume fractions of the mixture. 
Note that
$$
\bA_p=\bA_{pq}+\lambda_q(\bA_p-\bA_q),\quad \bA_q=\bA_{pq}-\lambda_p(\bA_p-\bA_q).
$$
This process is the basic step of the $(H_n)$-formalism (check references given in the Introduction), and produces successive mass-points and weights for intermediate probability measures $\{\mu_k\}$ as well as decomposition directions. The crucial ingredient for laminates is the rank-one condition for all decomposition directions.  

If we assume now that $\mu$ is coming from the discrete gradients of two continuous, scalar functions $(u, v)$ with constant gradients $(\nabla u_i, \nabla v_i)$ with respect to some and the same triangulation $\tau_l$ ($l$ is an arbitrary positive integer indicating level of discreteness)
\begin{equation}\label{inicial}
\mu=\sum_it_i\delta_{(\nabla u_i, \nabla v_i)},
\end{equation}
where will the decomposition in \eqref{decomposition} come from enabling for the $(H_n)$-formalism to be applied, guaranteeing the rank-one condition for all decomposition directions, and so be sure that $\mu$ is a laminate? Needless to say, we are seeking a general result to be valid for any such discrete gradient.

Our strategy is the following. Consider the two marginals
$$
\mu_u=\sum_it_i\delta_{\nabla u_i},\quad \mu_v=\sum_it_i\delta_{\nabla v_i}.
$$
Take any decomposition like \eqref{decompositiondos} for the first marginal $\mu_u$
\begin{equation}\label{paresz}
(s_1, \bu_1),\quad (s_2, \bu_2), \quad\dots\quad (s_m, \bu_m),\quad \{\bu_j\}=\{\nabla u_i\}.
\end{equation}
We can proceed with arbitrary mixtures for this set of pairs because the support of $\mu_u$ is contained in $\R^2$, and therefore the rank-one condition for decomposition directions is not an issue. Because $\nabla u_i$ and $\nabla v_i$ are linked to each other being the gradients of $u$ and $v$, respectively, in the same element of the triangulation $\tau_l$, we can consider an identical set of pairs as in \eqref{paresz} replacing $\bu_i$ by $\bv_i$ for all $i$ if $\{\bv_j\}=\{\nabla v_i\}$.
In doing so, we cannot expect in the least that corresponding decomposition directions in successive levels of the form
$$
\bU_j-\bU_k,\quad \bV_j-\bV_k,
$$
respectively for both sets of pairs, will be parallel to each other to ensure the rank-one condition for the matrix
$$
\begin{pmatrix}\bU_j-\bU_k\\\bV_j-\bV_k\end{pmatrix}.
$$
Note however that this rank-one condition is guaranteed for the first-level due to the rank-one compatibility across planar interfaces for gradients. 
We therefore envision the possibility of forcing this parallelism between decomposition directions, 
at the expense of allowing different volume fractions for the second marginal $\mu_v$
\begin{equation}\label{paresy}
(S_1, \bv_1),\quad (S_2, \bv_2), \quad\dots\quad (S_m, \bv_m),\quad \{\bv_j\}=\{\nabla v_i\}.
\end{equation}
The family of weights $\{S_j\}$ is to be selected in such a way that corresponding decomposition directions for both components when performing exactly the same mixtures at all levels for both sets of pairs in \eqref{paresz} and \eqref{paresy} turn out to be parallel. 

More specifically, imagine we could succeed in deciding on the following elements.
\begin{enumerate}
\item One concrete way, skeleton or mold, linked to the underlying triangulation, specifying 
what triangles are to be mixed at the first stage, and what mixtures of triangles are to be mixed among themselves in successive levels until the final mixture involving the complete mass of the unit cube $Q$. This mold must be independent of the particular gradients considered and depend only on the elements of $\tau_l$. It should be selected from the outset. We refer to it as $\skeleton$. 
\item A certain collection $\Theta_l$ of vectors of feasible weights that can be interpreted in a way to produce respective sets of pairs as in \eqref{paresz}, any time a function $w$ and its gradients $\{\nabla w_i\}$ with respect to $\tau_l$, are provided.
\end{enumerate}
Given two arbitrary functions, $u$ and $v$, $Q$-periodic, continuous, piece-wise affine with respect to $\tau_l$, we can define a (set-valued) operator $\bT$ from $\Theta_l$ to itself as follows. Let the vector of weights $\bt\in\Theta_l$ be given, and calculate the associated full set of decomposition directions along $\skeleton$ with the gradients $\{\nabla u_i\}$. Define the image $\bT(\bt)\subset\Theta_l$ as the collection of vectors of weights $\br\in\Theta_l$ that are capable of reproducing the same (parallel) decomposition directions along $\skeleton$ but calculated with the second set of gradients $\{\nabla v_i\}$. The existence of a fixed point for $\bT$, ensuring the same volume fractions for both components, will definitely be a proof that $\mu$ in \eqref{inicial} can be reached by lamination by the very way in which the operator $\bT$ has been defined. The arbitrariness of $l$ leads to our result. 

Our job consists in showing that this formal framework with the claimed features is possible. In addition to fixing notation, we need to define $\Theta_l$, the mold $\skeleton$, how a set of gradients $\{\nabla w_i\}$ enters into our discussion, and how to define the operator $\bT$  once $(u, v)$ are given and $\skeleton$ selected. 

The whole point of our approach is then to use Kakutani's fixed point theorem to prove the existence of at least one such fixed point for $\bT$.
Our main proofs then focus on showing that our analytical framework has been so defined as to allow for the hypotheses of Kakutani's theorem to be checked. This objective explains why we have to start with care for the description of those main ingredients $\skeleton$ and $\Theta_l$, and how they relate to each other. 

\section{$(H_n)$-conditions}\label{tres3}
Since these structures are central to rank-one convexity and laminates, they should play a fundamental role in our strategy. This justifies why we need to be more precise about the ingredients underlying these structures. Two basic sources where $(H_n)$-conditions are defined and treated are \cite{DacorognaE}, \cite{pedlam} (see also \cite{PedregalI}). Though they can be examined in a much more general context, we restrict attention specifically to our situation in this work. 

\begin{definition}\label{hn}
Let 
$$
\bbH=\{(s_i, \bw_i): 1\le i\le m\}
$$
be a finite collection of pairs such that
$$
s_i\ge0, \sum_i s_i=1, \bw_i\in\R^2,\quad \bcero=\sum_is_i\bw_i.
$$
A $(H_n)$-condition for $\bbH$ is a sequence of discrete probability measures 
$$
\{\mu_k: 1\le k\le m\},\quad \mu_k=\sum_i s_i^{(k)}\bw_i^{(k)}
$$ 
where
$$
\mu_m=\sum_is_i\bw_i, s_i^{(m)}=s_i, \bw_i^{(m)}=\bw_i, \quad \mu_{1}=\delta_{\bcero}, s_1^{(1)}=1, \bw_1^{(1)}=\bcero.
$$
Depending on whether we describe the passage $\mu_k\mapsto\mu_{k-1}$ or $\mu_{k-1}\mapsto\mu_k$, we talk, respectively, about the  bottom-to-top or top-to-bottom recursive procedure. 
\begin{itemize}
\item Bottom-to-top description $\mu_k\mapsto\mu_{k-1}$. It is determined recursively by the basic rule that consist in selecting two different indexes $1\le i_p\le i_q\le k$, and replacing the two pairs 
$$
(s^{(k)}_{i_p}, \bw^{(k)}_{i_p}),\quad (s^{(k)}_{i_q}, \bw^{(k)}_{i_q})
$$
in the support of $\mu_k$ by the single pair
\begin{equation}\label{pesosrel}
(s^{(k-1)}_{i_pi_q}, \bw_{i_pi_q}^{(k-1)}),\quad s^{(k-1)}_{i_pi_q}=s^{(k)}_{i_p}+s^{(k)}_{i_q},
\bw_{i_pi_q}^{(k-1)}=\frac{s^{(k)}_{i_p}}{s^{(k-1)}_{i_pi_q}}\bw^{(k)}_{i_p}+\frac{s^{(k)}_{i_q}}{s^{(k-1)}_{i_pi_q}}\bw^{(k)}_{i_q},
\end{equation}
in the support of $\mu_{k-1}$, and retaining unchanged the other pairs in the support of $\mu_k$ passing directly to the support of $\mu_{k-1}$. Once this operation has been performed, the resulting pairs in the support of $\mu_{k-1}$ are relabelled 
$$
\mu_{k-1}=\sum_is_i^{(k-1)}\delta_{\bw_i^{(k-1)}}
$$
to proceed recursively.
We simply say that pairs $i_p$ and $i_q$ have been mixed to pass from $\mu_k$ to $\mu_{k-1}$, and refer to the difference
\begin{equation}\label{dirdesz}
\bd_{i_qi_p}^{(k)}\equiv\bw^{(k)}_{i_q}-\bw^{(k)}_{i_p}
\end{equation}
as the corresponding decomposition direction. 
\item Top-to-bottom description $\mu_{k-1}\mapsto\mu_k$. We select one pair $(s_i^{(k-1)}, \bw_i^{(k-1)})$ in the support of $\mu_{k-1}$, and decide or are given the corresponding decomposition direction $\bd$, to play the role of the vector in \eqref{dirdesz}, and the relative volume fraction $\lambda\in(0, 1)$ as in \eqref{pesosrel}. Then the two pairs
$$
(s_i^{(k-1)}\lambda, \bw_i^{(k-1)}+(1-\lambda)\bd),\quad (s_i^{(k-1)}(1-\lambda), \bw_i^{(k-1)}
-\lambda\bd),
$$
become part of the support of $\mu_k$, while the other pairs pass unchanged from $\mu_{k-1}$ to $\mu_k$.
\end{itemize}
Both methods are, of course, equivalent. 
\end{definition}

In the bottom-to-top description, pairs in the support of $\mu_m$ are being mixed successively and at different levels, mixtures of mixtures (of mixtures), to produce the intermediate probability measures $\mu_k$ until reaching the final trivial measure $\mu_{1}$. If one is not interested in keeping track of which pairs are mixed at each iteration, couples in the support of each intermediate probability measure $\mu_k$ can be renumbered so that the first two pairs ($i_p=1$, $i_q=2$) are always the ones to be mixed at every iteration according to the basic rule. However, since it is especially important for us to specify the global way of mixing pairs across the family $\{\mu_k\}$, we avoid such reorganizations. Indeed, we will explicitly refer to this global way of passing from $\mu_m$ to $\mu_{1}$  through the intermediate probabilities $\mu_k$ as the mold or skeleton of the $(H_n)$-condition. In our situation here, there is no need to insist in that decomposition directions in \eqref{dirdesz} should be rank-one, as it is typically demanded, because the $\bw$'s are taken to be vectors, and hence decomposition directions are always rank-one. 

\begin{definition}
Let the finite set of pairs $\bbH$ as in Definition \ref{hn} be given. A mold or skeleton for a $(H_n)$-condition for $\bbH$ is a global way of specifying which pairs in the support of each probability $\mu_k$ are to be mixed to produce $\mu_{k-1}$ 
in the overall recursive process according to the previous definition. We will reserve the letter $\mathscr{S}$ to refer to such molds. 
\end{definition}

Typically, if initial pairs are labeled through subindex $i$, $1\le i\le m$, as in Definition \ref{hn}, $\mathscr{S}$ can be determined through the same set of indexes at successive levels using some delimiters, like brackets, to enclose pairs to be mixed as iterations proceed. 

\begin{example}
If we start with $\{1, 2, 3, 4, 5\}$ for a probability measure with a finite support of at most five vectors (note that some of those vectors could occur more than once in the enumeration of the starting set of pairs), then some possibilities for $\mathscr{S}$ are
\begin{gather}
[1, 3], 2, 4, 5\mapsto [[1, 3], 5], 2, 4\mapsto [[[1, 3], 5], 4], 2\mapsto [[[[1, 3], 5], 4], 2],\nonumber\\
1, [2, 4], 3, 5\mapsto 1, [2, 4], [3, 5]\mapsto 1, [[2, 4], [3, 5]]\mapsto [1,  [[2, 4], [3, 5]]],\nonumber\\
1, 2, [3, 5], 4\mapsto 1, [2, [3, 5]], 4\mapsto [1, 4], [2, [3, 5]]\mapsto [[1, 4], [2, [3, 5]]].\nonumber
\end{gather}
The last case can also be implemented in the form
$$
1, 2, [3, 5], 4\mapsto [1, 4], 2, [3, 5]\mapsto [1, 4], [2, [3, 5]]\mapsto [[1, 4], [2, [3, 5]]], 
$$
or in short
$$
[1, 4], 2, [3, 5]\mapsto [1, 4], [2, [3, 5]]\mapsto [[1, 4], [2, [3, 5]]].
$$
The second example above can also be expressed as
$$
1, [2, 4], [3, 5]\mapsto 1, [[2, 4], [3, 5]]\mapsto [1,  [[2, 4], [3, 5]]].
$$
It should be clear what is meant by asserting that the mold $\mathscr{S}$ is given by, for example, 
$$
1, [2, 4], [3, 5]\mapsto 1, [[2, 4], [3, 5]]\mapsto [1,  [[2, 4], [3, 5]]]
$$ 
for a set of (unspecified) five pairs weights/vectors:
\begin{enumerate}
\item first, pairs $2$ and $4$ are mixed according to the basic rule in Definition \ref{hn}; pairs $3$ and $5$ can be simultaneously transformed because participating pairs in these two groups are disjoint;
\item then, the resulting pairs coming from those two first mixtures are in turn mixed recursively in a second level, by using the same basic, transformation rule; 
\item finally, the first initial pair is mixed with the resulting pair from the previous item. 
\end{enumerate}
Note how $\mathscr{S}$ is independent of what the specific pairs are: every time we feed five true pairs weights/vectors to mold $\mathscr{S}$, a specific $(H_n)$-condition is built with a full set of corresponding decomposition directions \eqref{dirdesz}, one for each mixture that is implemented.
\end{example}

Let us stress this issue. Given the initial set of pairs $\bbH$, each possible mold $\mathscr{S}$ for $\bbH$ would give rise to a different $(H_n)$-condition; conversely, given a fixed mold $\mathscr{S}$ for a certain number $m$ of pairs, corresponding $(H_n)$-conditions are generated every time we feed $\mathscr{S}$ with a real set of pairs weights/vectors 
$$
\bbH=\{(s_i, \bw_i): 1\le i\le m\}.
$$

The important point is to realize that given the initial set of pairs $\bbH$ and mold $\mathscr{S}$, the full $(H_n)$-condition is completely determined by those two elements, and, in particular, all decomposition directions. In this way, if we are given two sets of pairs $\bbH_1$ and $\bbH_2$ with the same number of ordered elements, the same mold $\mathscr{S}$ can be used with both, as indicated above, to produce two similar $(H_n)$-conditions. We can, therefore, compare the full set of decomposition directions across levels produced for both $\bbH_1$ and $\bbH_2$, and check whether they are parallel or not. 

It must be clear by now that there is a great number of possible molds $\mathscr{S}$ for a given initial set $\bbH$ of $m$ pairs, especially if $m$ is large. The example above makes us realize that it is not easy to describe such molds in a coherent way in full generality. 
There is, however, one particular class of molds which admit a more transparent recursive structure. They are given in a dyadic manner. 

\begin{definition}
A mold $\mathscr{S}$ is a dyadic mold in $m$ levels for an initial set $\bbH$ of $2^m$ elements, if mixtures are implemented in a complete, pair-wisely manner at each level, before proceeding to the next. 
\end{definition}

A simple example will clarify this concept.

\begin{example}
Take $\bbH$ with $8=2^3$ elements $\{1, 2, 3, 4, 5, 6, 7, 8\}$. A possible dyadic mold could be
$$
[1, 5], [2, 8], [3, 4], [6, 7]\mapsto [[1, 5], [6, 7]], [[2, 8], [3, 4]]\mapsto [[[1, 5], [6, 7]], [[2, 8], [3, 4]]].
$$
Another possibility is
$$
[1, 6], [2, 3], [4, 7], [5, 8]]\mapsto [[1, 6], [4, 7]], [[2, 3], [5, 8]]\mapsto [[[1, 6], [4, 7]], [[2, 3], [5, 8]]].
$$
\end{example}

\begin{proposition}\label{permutacion}
Given a dyadic mold $\mathscr{S}$ in $m$ levels, there is always a unique permutation of the set $\{1, 2, \dots, 2^m\}$ such that mixtures always take place in the form $[2i-1, 2i]$ in all levels.
\end{proposition}
The proof is elementary. In the second case of the previous example the permutation is defined through the identification
$$
[[[1, 6], [4, 7]], [[2, 3], [5, 8]]]\mapsto [[[1, 2], [3, 4]], [[5, 6], [7, 8]]].
$$

Because of this proposition, there is no loss of generality in describing $(H_n)$-conditions organized by a dyadic mold in the following way. 
\begin{enumerate}
\item Initialization. We identify in the vector
$$
\bt=(t_1, t_2, \dots, t_{2^m})
$$
our initial set of weights for the $(H_n)$-condition, and put
$$
t_k^{(m)}=t_k,\quad \bw_k^{(m)}\in\R^2,
$$
for elements to initialize the $(H_n)$-condition at the lowest level.
\item Recursion. 
\begin{enumerate}
\item Weights for new level, and relative weights. For 
$$
p=m-1, m-2, \dots, 1, 0,\quad 1\le k\le 2^p,
$$ 
put
$$
t_k^{(p)}=t_{2k-1}^{(p+1)}+t_{2k}^{(p+1)},
$$
and
\begin{equation}\label{pesosz}
\lambda_k^{(p)}=\begin{cases}
\frac{t_{2k}^{(p+1)}}{t_k^{(p)}},&  t_k^{(p)}>0\\
1/2,&  t_k^{(p)}=0\end{cases}.
\end{equation}
In this way
\begin{equation}\label{relabs}
t_{2k}^{(p+1)}=t_k^{(p)}\lambda_k^{(p)},\quad t_{2k-1}^{(p+1)}=t_k^{(p)}(1-\lambda_k^{(p)}),
\end{equation}
and $t_1^{(0)}=1$. 
\item Decomposition direction. For 
$$
p=m-1, m-2, \dots, 1, 0,\quad 1\le k\le 2^p,
$$ 
define
\begin{equation}\label{dirdes}
\bW_k^{(p)}=\bw_{2k-1}^{(p+1)}-\bw_{2k}^{(p+1)},
\end{equation}
as the corresponding decomposition direction. 
\item Mass points for new level. For 
$$
p=m-1, m-2, \dots, 1, 0,\quad 1\le k\le 2^p,
$$ 
set
\begin{equation}\label{equis}
\bw_k^{(p)}=(1-\lambda_k^{(p)})\bw_{2k-1}^{(p+1)}+\lambda_k^{(p)}\bw_{2k}^{(p+1)}.
\end{equation}
\end{enumerate}
\end{enumerate}
This recursive procedure is repeated $m$ times until one gets one unique final vector $\bw_1^{(0)}$. 
Decomposition directions $\bW_k^{(p)}$, vectors $\bw_k^{(p)}$,
and relative weights $\lambda_k^{(m-1)}$ as well, depend upon $\bt$, the initial vector of weights, and on $\{\bw^{(m)}_k\}$, the initial set of vectors. But regarding this set of vectors as given, we will make the dependence on $\bt$ explicit by simply putting
\begin{equation}\label{elementos}
\bW_k^{(p)}(\bt), \quad\bw_k^{(p)}(\bt), \quad\lambda_k^{(m-1)}(\bt).
\end{equation}

This is the bottom-to-top description. Equivalently, we can adopt the top-to-bottom procedure which starting from the trivial probability measure and pair
$$
\mu_1=\delta_{\bw_1^{(0)}},\quad (1, \bw_1^{(0)}),
$$ 
and given the net of decomposition directions
and relative volume fractions
$$
(\bW_k^{(p)}, \lambda_k^{(p)}),\quad p=0, 1, 2, \dots, m-1, 1\le k\le 2^p,
$$
we can build, in a successive dyadic manner, the intermediate probability measures $\mu_p$ until the final $\mu_m$. 

There is finally another special class of molds for $(H_n)$-conditions which will be important for. They build upon dyadic sub-molds. 
\begin{definition}\label{tridiadico}
Suppose we have a set of pairs $\bbH$ with $3\times 2^m$ elements for some $m\in\N$. We will say that a mold $\mathscr{S}$ for $\bbH$ is tri-dyadic if mixtures corresponding to the three subsets
$$
\{1, 2, \dots, 2^m\},\quad \{2^m+1, 2^m+2, \dots, 2\times 2^m\},\quad \{2\times2^m+1, 2\times2^m+2, \dots, 3\times 2^m\},
$$
are not mixed among themselves until they all are exhausted individually in a dyadic manner.
\end{definition}

To illustrate this concept another explicit example may help.
\begin{example}
Take the set of indexes 
$$
\{1, 2, \dots, 8, 8+1, 8+2, \dots, 8+8, 8+8+1, 8+8+2, \dots, 8+8+8\},
$$
i.e.
$$
\{1, 2, \dots, 8 / 9, 10, \dots, 16 / 17, 18, \dots, 23, 24\}.
$$
A tri-dyadic mold proceed independently in each subset
$$
\{1, 2, \dots, 8\},\quad \{8+1, 8+2, \dots, 8+8\},\quad \{8+8+1, 8+8+2, \dots, 8+8+8\}
$$
in a dyadic form like, for example,
\begin{gather}
[1, 5], [2, 8], [3, 4], [6,7]\mapsto [[1, 5], [6, 7]], [[2, 8], [3, 4]]\mapsto [[[1, 5], [6, 7]], [[2, 8], [3, 4]]],\nonumber\\
[9, 14], [10, 16], [11, 15], [12, 13]\mapsto [[9, 14], [11, 15]], [[10, 16], [12, 13]]\mapsto \nonumber\\
[[[9, 14], [11, 15]], [[10, 16], [12, 13]]],\nonumber\\
[17, 19], [18, 22], [20, 24], [21, 23]\mapsto [[17, 19], [21, 23]], [[18, 22], [20, 24]]\mapsto\nonumber\\
[[[17, 19], [21, 23]], [[18, 22], [20, 24]]].\nonumber
\end{gather}
Notice the three levels of mixture for the three cases. 
Finally, the resulting mixtures in the three processes
\begin{gather}
[[[1, 5], [6, 7]], [[2, 8], [3, 4]]], \nonumber\\
[[[9, 14], [11, 15]], [[10, 16], [12, 13]]], \nonumber\\
[[[17, 19], [21, 23]], [[18, 22], [20, 24]]],\nonumber
\end{gather}
are treated or mixed in any fixed, specific way to reach the final trivial measure. 
\end{example}

\section{$(H_n)$-conditions for gradients: molds through triangulations}\label{dosz}
We will be working with $Q$-periodic, continuous, piecewise affine, two-component maps with respect to a specific family of triangulations $\{\tau_l\}$ of the unit cube $Q$ of $\R^2$. In general, the unit cube $Q\subset\R^N$ can be decomposed in a finite number of simplexes and with a finite number $d(N)$ of normals to the flat faces of those simplexes. 
By making small copies of $Q$ and making use of this decomposition in all copies, we can build a family of exhausting, regular triangulations that provide arbitrary, uniform approximations of Lipschitz functions by piecewise affine maps. This is standard and well-known (see, for instance, \cite{ekelandtemam}). 
For $N=2$, three normals suffice, while for dimension $N=3$, seven are necessary, and so on. Because of the difficulties organizing elements, normals, edges, vertices (much in the same way as it is done in the analysis of finite elements), we restrict attention to $N=2$. Higher values of $N$ would require further insight into that way of organizing things. 
For future reference, we will put $\bbP_l$ for the finite-dimensional space of $Q$-periodic, $\tau_l$-piece-wise linear, continuous functions with gradients that are constant on each element of $\tau_l$. 

The description that follows pretends to setup a way to link triangulations $\tau_l$ with the $(H_n)$-hierarchy without reference to specific gradients or to volume fractions. It is just a way to explicitly tag, through the triangles of $\tau_l$, their planar interfaces, and their unions in a hierarchical form, how gradients and volume fractions are to be mixed once they are provided. Said differently, we would like to determine a mold or skeleton $\skeleton$ associated with triangulation $\tau_l$ that will be fed with volume fractions (weights) and gradients (vectors) of functions in $\bbP_l$. 

Before going into the general case for arbitrary values for $l$, we will deal with a particular case to facilitate the understanding of what we are trying to convey. We will take $l=2$. Notice that this particular case would correspond to the example in \cite{sebsze}. 

Suppose we have a triangulation $\tau_2$ of the unit cube $Q$ in $16(=2^2\times2^2)$ small, equal squares, and each one  is divided into two triangles along one of the two diagonals (always the same). We will write
$$
\tau_2=\{T_{s, 2}: s\in\{1, 2, \dots, 32\}\}
$$
for the full collection of triangle of $\tau_2$. Since mixtures of gradients across planar interfaces between adjacent elements in $\tau_l$ must play a fundamental role, our mold will hinge on those in the following way. 

For $192=3\times 2^6$, we want to describe a suitable tri-dyadic mold $\skeletondos$ organized through the three subsets of indexes
\begin{equation}\label{tresz}
I_1\equiv\{1, 2, \dots, 64\},\quad I_2\equiv\{65, 66, \dots, 128\},\quad I_3\equiv\{129, \dots, 192\},
\end{equation}
which will be identified, in any given way, with the three possible subsets of two normals out of the set of three normals $(1, 0), (0, 1), (1, 1)$ utilized in $\tau_l$, for instance
\begin{gather}
j=1\mapsto\hbox{normals }(1, 0), (0, 1)\mapsto I_1,\nonumber\\
j=2\mapsto\hbox{normals }(1, 0), (1, 1)\mapsto I_2,\label{normalesz}\\ 
j=3\mapsto\hbox{normals }(0, 1), (1, 1)\mapsto I_3.\nonumber
\end{gather}
Since planar interfaces in $\tau_l$ furnish compatibility of gradients across it, our mold will hinge on the full set of those interfaces. 
Take $j=1$ to begin with, and focus on
$$
I_1=\{1, 2, \dots, 64\}, \quad\hbox{ normals }(1, 0), (0, 1).
$$
Each planar interface with either normal $(1, 0)$ or $(0, 1)$ in $\tau_2$, the normals corresponding to $j=1$, will be assigned an index $i$, $1\le i\le 32$, in an arbitrary manner, but then $2i$ and $2i-1$ should be the indexes in $I_1$ assigned to the adjacent elements of $\tau_2$ corresponding to interface $i$. Note that this chosen, specific way of tagging planar interfaces with normals $(1, 0)$ or $(0, 1)$ in $\tau_2$ with index $i$ will uniquely determine how subsequent mixtures in later upper levels will be performed as the mixture rule $[2p-1, 2p]$ will always be applied in all levels as it should in a dyadic mold. The arbitrariness of this choice implies that any such manner of labelling planar interfaces is equally valid. 

Proceed similarly with $j=2$, and $j=3$. This means that for $j=2$ planar interfaces with normals $(1, 0)$ or $(1, 1)$ 
will be identified with indexes $i$, $33\le i\le 64$, in a manner that $2i$ and $2i-1$ in $I_2$ are the indexes assigned to the two adjacent elements of $\tau_2$ with the corresponding interface $i$; and similarly with $j=3$. Thus all planar interfaces will be assigned two indexes because corresponding normals occur twice in our tag method according to \eqref{normalesz}, and triangles in $\tau_l$ will be given six different indexes, two for each $j$. Hence, 
this procedure defines a partition of the set of indexes $\{1, 2, \dots, 192\}$ in $32$ subsets $\C_s$ of six elements
\begin{equation}\label{sumazz}
\C_s=\{k_{1s}, k_{2s}, k_{3s}, k_{4s}, k_{5s}, k_{6s}\},\quad k_{is}\in\{1, 2, \dots, 192\},\quad \cup_{1\le s\le 32}\C_s=\{1, 2, \dots, 192\},
\end{equation}
in such a way that the intersection of each $\C_s$ with the three subsets in \eqref{tresz}
is exactly two indexes. Index $s$ identifies elements in $\tau_2$, and so \eqref{sumazz} implies, as just indicated above, that each triangle in $\tau_2$ is identified by six different indexes of $\{1, 2, \dots, 192\}$, two for each of the $I_j$'s, and such indexes are compatible with the labelling of planar interfaces as described in the previous paragraph.

Our tridyadic mold $\skeletondos$ will proceed simultaneously and independently for all $j$ in \eqref{normalesz}, according to Definition \ref{tridiadico}, mixing adjacent elements $2i$ and $2i-1$ through interface $i$, in the first level of mixture; and in the associated dyadic manner $[2p-1, 2p]$ inherited by the chosen tag mechanism for interfaces in all upper levels. Once the three independent dyadic processes have been exhausted, fix any form to mix the three final resulting delta measures using again the basic mixing rule of the $(H_n)$-condition formalism. 

The general case for any value of $l$ is described exactly in the same manner. We explain it in shorter terms. 

We fix a family of triangulations $\{\tau_l\}$, $l=1, 2, \dots$,  in which the two sides of the unit cube $Q$ of $\R^2$ parallel to the coordinate axes are divided in $2^l$ subintervals producing $2^{2l}$ identical sub-squares of side-length $2^{-l}$.  Each such resulting sub-square is divided into two triangles along the diagonal with normal parallel to the vector $(1, 1)$, for instance. In this way we can write
$$
\tau_l=\{T_{s, l}: 1\le s\le 2^{2l+1}\}, \quad T_{s, l}, \hbox{ triangle of $\tau_l$ with area $2^{-1-2l}$}.
$$
After our experience with the case $l=2$, a partition in $2^{2l+1}$ pair-wise disjoint subsets $\C_s$, $1\le s\le 2^{2l+1}$, as many as triangles $T_{s, l}$ in $\tau_l$, of the set of indexes $\{1, 2, \dots, 3\times2^{2l+2}\}$ is possible complying with the following conditions:
\begin{itemize}
\item Each $\C_s$ has six elements, and it should have two indexes in each one of the subsets
\begin{gather}
I_1=\{1, 2, \dots, 2^{2l+2}\},\quad I_2=\{2^{2l+2}+1, 2^{2l+2}+2, \dots, 2\times2^{2l+2}\},\label{subcon}\\ 
I_3=\{2^{2l+2}+1, 2^{2l+2}+2, \dots, 3^{2l+2}\}.\nonumber
\end{gather}
\item $\{\C_s\}$ must be compatible with labels $i$ for planar interfaces in $\tau_l$ exactly in the same way as in the case $l=2$, so that mixtures in a first level are identified with planar interfaces $i$ and corresponding adjacent triangles $2i$ and $2i-1$.  The same dyadic rule $[2p-1, 2p]$ is used throughout all levels of mixture.
\item Each planar interface in $\tau_l$ is tagged with two indexes, and each element with six indexes in a compatible manner for the dyadic mold to be applied. 
\end{itemize}
There is, consequently, an underlying tri-dyadic skeleton $\skeleton$ that will be fed with vectors of weights and constant gradients of functions in $\bbP_l$.  For each $l$, we fix once and for all the tridyadic mold $\skeleton$ among the many possibilities respecting the conditions just described. 

\begin{definition}\label{molde}
Let $l\in\N$, and $\tau_l$, a corresponding uniform triangulation of the unit cube $Q\subset\R^2$. A tri-dyadic mold $\skeleton$ for $\tau_l$ is specified through a coherent labelling of planar interfaces and elements of $\tau_l$ as has been described above. Every such mold $\skeleton$ is associated with a suitable partition $\{\C_s\}$ of the triangles of $\tau_l$ in which each element occurs six times. 
\end{definition}

There are no vectors or functions in this discussion. As has been conveyed in Section \ref{tres3}, 
the chosen tridyadic mold $\skeleton$ will be utilized to produce true $(H_n)$-conditions and associated decomposition directions, as soon as we feed into $\skeleton$ the gradient $\nabla w$ of a function $w\in\bbP_l$, and a vector of weights coming from a suitable set. 
Recall that $\bbP_l$ is taken to be the finite-dimensional space of $Q$-periodic, $\tau_l$-piece-wise linear, continuous functions with gradients that are constant on each $T_{s, l}$.
For later reference, we will designate by $\skeletonj$ the sub-mold associated with the independent processes for each $j=1, 2, 3$. Schematically we can write 
\begin{equation}\label{parcial}
\{\skeletonj: j=1, 2, 3\}\mapsto \skeleton
\end{equation}
to stress that each dyadic $\skeletonj$ is a part of the tridyadic mold $\skeleton$, and that the $\skeletonj$'s proceed independently of each other. In particular, all decomposition directions generated through $\skeleton$ any time weights and gradients are fed, will include those generated by each $\skeletonj$. 
Recall that $\skeleton$ and $\skeletonj$ are fixed once and for all. 

\section{The role of weights and functions}\label{uno}
As it was pointed out several times in Section \ref{tres3}, once the tri-dyadic mold has been selected, it should be fed with weights and vectors to produce true $(H_n)$-conditions.

We start with the description of a suitable set of weights. We again treat first the case $l=2$ to make the discussion more transparent. 

We would like to consider a certain suitable subset $\Theta_2$ of vectors of weights of the standard simplex $\Gamma_{192}$ of $\R^{192}$, associated with the chosen $\skeletondos$ and its corresponding partition $\C_s$ according to Definition \ref{molde}, namely
\begin{align}
\Theta_2=\{\bt=(t_1, t_2, &\dots, t_{192})\in\R^{192}: t_k\ge0, 1\le k\le 192, \label{dominioa}\\
&t_{2i}+t_{2i-1}=1/96, 1\le i\le 96, \label{interfasea}\\
&\sum_{k\in\C_s}t_k=1/32, 1\le s\le 32\}.\label{sumaa}
\end{align}
As in Section \ref{dosz}, bear in mind that:
\begin{enumerate}
\item Index $s$  in  \eqref{sumaa} identifies triangle $T_{s, 2}$ in $\tau_2$. The total mass $1/32$ of each such triangle is divided into six pieces according to the weights $t_{k_{is}}$ in \eqref{sumazz}.
\item Each index $i$ in \eqref{interfasea} identifies one of the $48$ planar interface in $\tau_2$ (recall periodicity). Such identification is made according to the rule in Section \ref{dosz} concerning flat interfaces in $\tau_2$ in each of the subsets $I_j$. 
Each such interfase occurs twice (corresponding to two different indexes $i$), and that is why index $i$ runs from 1 to 96. Condition \eqref{interfasea} furthermore ensures that mixtures coming from planar interfaces of $\tau_2$ will always be made in a fixed proportion $1/96$. There is a special reason for this restriction. 
\end{enumerate}

\begin{definition}\label{pesos}
Let $l\in\N$, and $\tau_l$, a regular, uniform partition of the unit cube $Q\subset\R^2$. Let $\skeleton$ be a specific, tri-dyadic mold associated with $\tau_l$. Let $\C_s$ be the corresponding partition of the elements of $\tau_l$. 
We put
\begin{align}
\Theta_l=\{\bt=(t_1, t_2, &\dots, t_{3\times2^{2l+2}})\in\R^{3\times2^{2l+2}}: t_k\ge0, 1\le k\le 3\times2^{2l+2}, \label{dominio}\\
&t_{2i}+t_{2i-1}=3^{-1}2^{-1-2l}, 1\le i\le 3\times2^{2l+1}, \label{interfase}\\
&\sum_{k\in\C_s}t_k=2^{-1-2l}, 1\le s\le 2^{2l+1}\},\label{suma}
\end{align}
a non-empty, convex, compact subset of the standard simplex 
$$
\Gamma_{3\times2^{2l+2}}=\left\{\bt=(t_1, t_2, \dots, t_{3\times2^{2l+2}})\in\R^{3\times2^{2l+2}}: t_k\ge0, \sum_kt_k=1\right\}.
$$
\end{definition}

The link of the set $\Theta_l$ with our triangulation $\tau_l$, mold $\skeleton$, and partition $\C_s$ should be easy to grasp as it has just been anticipated for the case $l=2$. 
\begin{enumerate}
\item Index $s$  in \eqref{suma} identifies triangles in $\tau_l$.
\item Each $i$ in \eqref{interfase} identifies a planar interface in $\tau_l$, and each planar interface corresponds to two different values of $i$, in such a way that triangles sharing a flat interface associated with index $i$ are to be mixed in proportions $t_{2i}$ and $t_{2i-1}$, and they all are assigned a similar weight $3^{-1}2^{-1-2l}$ every time such interface is utilized in the mold $\skeleton$. 
\end{enumerate}

Concerning functions, let $w\in\bbP_l$, and consider its associated probability measure
\begin{equation}\label{gradiente}
\nu_w=\sum_s2^{-2l-1}\delta_{\left.\nabla w\right|_{T_{s, l}}}
\end{equation}
Recall that $\tau_l=\{T_{s, l}\}$, and hence  each $\left.\nabla w\right|_{T_{s, l}}$ is a constant vector in $\R^2$. The $3\times2^{2l+2}$ vectors to feed, together with weights in $\Theta_l$, the tri-dyadic mold $\skeleton$ (with partition $\C_s$) are
$$
 \bw_k=\left.\nabla w\right|_{T_{s, l}}\hbox{ if }k\in\C_s,\quad k\in\{1, 2, \dots, 3\times2^{2l+2}\}.
 $$

According to Definition \ref{tridiadico}, the $(H_n)$-conditions corresponding to sub-molds $\skeletonj$ proceed, separately for each $j$, in a dyadic manner as it has been explicitly described in Section \ref{tres3}.  If we write
\begin{equation}\label{tresjotas}
\bt=(\bt_1, \bt_2, \bt_3)\in\Theta_l,\quad \bt_j=(t^{(j)}_k), \quad 1\le k\le 2^{2l+2}, j=1, 2, 3, 
\end{equation}
we proceed as in Section \ref{tres3} in a dyadic manner for each $\bt_j$, though there are a couple of particular issues that need to be explicitly emphasized. 
We use the same notation utilized in the discussion right after Proposition \ref{permutacion}, though, for the sake of simplicity, we ignore (super)-index $j$, since the process is formally the same for the three dyadic sub-molds. 
\begin{enumerate}
\item Initialization. As just stressed, we would 
work separately on each $I_j$ in \eqref{subcon}, ignoring the index $j$ for the sake of notational simplicity, and take
\begin{equation}\label{soporte}
t_k^{(m)}=t_k,\quad \bw_k^{(m)}=\left.\nabla w\right|_{T_{s, l}}\hbox{ if }k\in\C_s,
\end{equation}
 for $1\le k\le 2^{2l+2}$. 
\item Recursion. The only point to be noticed is that in our particular situation, because the way in which the set $\Theta_l$ has been defined (recall \eqref{interfase}), we would have
\begin{equation}\label{pesofijo}
t_k^{(m-1)}=2^{1-m}, \quad 1\le k\le 2^{m-1},
\end{equation}
and then
\begin{equation}\label{constante}
\lambda_k^{(p)}=1/2,\quad 1\le p\le m-2, 1\le k\le 2^p.
\end{equation}
\end{enumerate}
All other ingredients of the discussion after Proposition \ref{permutacion} are exactly the same, including the top-to-bottom description, and we will use the same notation 
\begin{equation}\label{elementos}
\bW_k^{(p)}(\bt), \quad\bw_k^{(p)}(\bt), \quad\lambda_k^{(m-1)}(\bt)
\end{equation}
for decomposition directions, mass-points of intermediate probability measures, and relative weights. 
Keep in mind that $\lambda_k^{(p)}=1/2$ for all 
$$
0\le p\le m-1,\quad 1\le k\le 2^p,
$$ 
according to \eqref{constante}.

Finally, to complete the $(H_n)$-condition with the tri-dyadic mold $\skeleton$, we need to keep working with the result of the three independent dyadic processes $\skeletonj$, according to what has been decided in $\skeleton$ until we reach the upper, final vanishing barycenter $\bcero\in\R^2$. Due to periodicity, we should have
\begin{equation}\label{baricentro}
\sum_s2^{-2l-1}\left.\nabla w\right|_{T_{s, l}}=\bcero,
\end{equation}
and we write 
\begin{equation}\label{baricentros}
\bbW_j=\frac1{1/3}\sum_s\sum_{k \in\C_s\cap I_j} t^{(j)}_k \left.\nabla w\right|_{T_{s, l}},\quad \bt_j=(t^{(j)}_k), \quad j=1, 2, 3, \quad \bt=(\bt_1, \bt_2, \bt_3)\in\Theta_l,
\end{equation}
for the barycenters of each of the three dyadic sub-molds $\skeletonj$ which are combined in two additional final mixtures in an arbitrary but fixed way. 
Since
$$
\frac13=\sum_s\sum_{k \in\C_s\cap I_j} t^{(j)}_k
$$
then
\begin{equation}\label{firstcero}
\bcero=\frac13(\bbW_1+\bbW_2+\bbW_3).
\end{equation}
It is convenient to use the same recursive notation for the whole mold $\skeleton$ as has been utilized above for each dyadic $\skeletonj$. In this way, we will indicate
\begin{gather}
t_k^{(m)}=t_k,\quad k\in\{1, 2, \dots, 3\times2^{2l+2}\}, \quad \bt=(t_k)\in\Theta_l,\nonumber\\
\bw_k^{(m)}=\left.\nabla w\right|_{T_{s, l}},\quad k\in\C_s,\nonumber
\end{gather}
for the initial elements to feed $\skeleton$. 
The only difference is that now we have $2l+4$ levels: $2l+2$ for the three dyadic processes, and two additional ones for the last mixture involving the $\bbW_j$'s in \eqref{firstcero}. \eqref{elementos} will also indicate decomposition directions, mass points for intermediate probability measures, and volume fractions for $m=2l+4$ successive levels in $\skeleton$. 

\section{The map}\label{tres}
Take an arbitrary value of $l$. Let $(u, v):Q\to\R^2$ be an arbitrary, two-component vector map of functions $u$ and $v$ in $\bbP_l$. We turn to our mold $\skeleton$, as has been fixed, to produce $(H_n)$-conditions when fed with weights from $\Theta_l$ and values of gradients $\nabla u$ or $\nabla v$. As just indicated above, $m=2l+4$ is the overall number of levels for $\skeleton$. Recall notation in \eqref{elementos} for decomposition directions, mass points, and relative volume fractions. 

A joint, simultaneous rank-one description of the joint probability measure
\begin{equation}\label{margi}
\nu=\sum_s2^{-2l-1}\delta_{\left.(\nabla u, \nabla v)\right|_{T_{s, l}}},
\end{equation}
demands, as a fundamental ingredient, that all decomposition directions  
\begin{equation}\label{dirdescomposicion}
\bU_k^{(p)}(\bt),\quad \bV_k^{(p)}(\bt)
\end{equation}
be proportional to each other for some and the same vector $\bt$ of weights, and for all 
$$
0\le p\le m-1,\quad 1\le k\le 2^p.
$$ 
Clearly for a given feasible vector $\bt\in\Theta_l$, those two sets of decomposition directions  in \eqref{dirdescomposicion} generated by the two components $u$ and $v$ will, most of the time, not be proportional to each other though
it is guaranteed for $p=m-1$ because decomposition directions at this first lower level are taken to be precisely the three normals occurring in $\tau_l$ (for both components $u$ and $v$). Yet one may wonder if one could ensure such parallelism for all levels, enabling different vector of weights: once fixed $u$ and $v$, for each $\bt\in\Theta_l$, are there feasible vectors of weights $\br$ such that  
$$
\bU_k^{(p)}(\bt),\quad\bV_k^{(p)}(\br)
$$ 
are parallel? The definition of our operator is in charge of guaranteeing this property.

Manipulations performed in the preceding section for our $\skeleton$ and for arbitrary $w\in\bbP_l$ and $\bt\in\Theta_l$ can be exactly repeated separately for both $u$ and $v$, respective vectors of weights $\bt$ and $\br$ in $\Theta_l$, corresponding probability measures 
$$
\nu_u=\sum_s2^{-2l-1}\delta_{\left.\nabla u\right|_{T_{s, l}}}, \quad \nu_v=\sum_s2^{-2l-1}\delta_{\left.\nabla v\right|_{T_{s, l}}},
$$
and initial elements
\begin{gather}
t_k^{(m)}=t_k,\quad r_k^{(m)}=r_k, \quad k\in\{1, 2, \dots, 3\times2^{2l+2}\},\nonumber\\
\bu_k^{(m)}=\left.\nabla u\right|_{T_{s, l}},\quad \bv_k^{(m)}=\left.\nabla v\right|_{T_{s, l}},\quad k\in\C_s.\nonumber
\end{gather}
The outcome of such a procedure would be the respective sets of decomposition directions, and the sets of mass points for intermediate probability measures
\begin{equation}\label{descomposicion}
\bU_k^{(p)}(\bt), \quad\bu_k^{(p)}(\bt),\quad 
\bV_k^{(p)}(\br), \quad\bv_k^{(p)}(\br),\quad 0\le p\le m-2, 1\le k\le 2^p.
\end{equation}
The families of weights $\lambda_k^{(m-1)}(\bt)$, $\lambda_k^{(m-1)}(\br)$ are however independent of $u$ or of $v$. 

Recall again that the mold $\skeleton$ has already been selected once and for all as described in Section \ref{dosz}, with a global number $m=2l+4$ of levels. 

\begin{definition}\label{mapeo}
Given $u, v\in\bbP_l$, we define the set-valued map
$$
\bT\equiv\bT_{l, \skeleton, u, v}:\Theta_l\mapsto\Theta_l,
$$
by putting
$$
\bT(\bt)=\{\br\in\Theta_l: \bV_k^{(p)}(\br)\parallel\bU_k^{(p)}(\bt)\hbox{ for all }0\le p\le m-2, 1\le k\le 2^p\}.
$$
\end{definition}
Note again how $\bV_k^{(m-1)}(\br)$, for $p=m-1$, is always parallel to $\bU_k^{(m-1)}(\bt)$ precisely because decomposition directions at the level $p=m-1$ correspond to interfaces between two adjacent elements of the triangulation $\tau_l$. 

The whole point or our concern is the following.
\begin{proposition}\label{facil}
The gradient measure $\nu$ in \eqref{margi} is a laminate if there is a fixed point for $\bT$, i.e. if there is $\bt\in\Theta_l$ such that $\bt\in\bT(\bt)$. 
\end{proposition}
\begin{proof}
The proof is immediate after our above discussion, and the definition of our operator $\bT$. 
\end{proof}

We will be using the following classic result to show the existence of a fixed-point for $\bT$. 

\begin{theorem}\label{kakutani} (Kakutani's fixed point theorem) 
Let $\A\subset\R^d$ be a non-empty, compact, convex set, and let $\F:\A\mapsto \A$ be an upper semicontinuous, set-valued map with non-empty, convex, compact values. Then $\F$ has a fixed point; that is, there is $\bx\in \A$ with $\bx\in\F(\bx)$. 
\end{theorem}
This is a classical theorem on fixed-points for set-valued maps, which is but a generalization of the classic Brower's fixed point theorem. It is well-known and can be found in many places, for instance in \cite{smirnov}. 

The fundamental properties that the application of this result to our framework require are the non-emptiness, the compactness and the convexity of  $\bT(\bt)$ for each $\bt\in\Theta_l$, in addition to the upper semicontinuity. Once these properties are shown, we will have our main result Theorem  \ref{objetivo} as a direct consequence of Proposition \ref{facil}, Theorem \ref{kakutani}, and the standard approximation fact for Lipschitz functions mentioned at the beginning of Section \ref{dosz}. 

We proceed to show first those basic requirements for map $\bT$, while we defer the non-emptiness of images to a final independent section.

\subsection{Basic properties of the map}\label{cuatro}
We start with the upper semicontinuity required by Theorem \ref{kakutani}. 
This property is, as a matter of fact, elementary since if 
$$
\br_n\in\bT(\bt_n),\quad \br_n\to \br, \bt_n\to \bt,
$$
then, we must necessarily have $\br\in\bT(\bt)$. This is straightforward because the dependence of elements in \eqref{elementos} on $\bt$ is continuous. 

On the other hand, the compactness of each subset $\bT(\bt)$ is also clear since all these images are closed subsets of the compact set $[0, 1]^q$ for some finite $q$. 

We treat next the convexity of images. 
Ensuring this convexity property is responsible for the precise definition of the set $\Theta_l$ we have adopted, and the way in which $(H_n)$-conditions have been setup in Sections \ref{tres3} and \ref{uno}. It is pretty clear after the following statement.
\begin{proposition}\label{convexidad}
\begin{enumerate}
\item For 
$$
p=m-1, m-2, \dots, 1,\quad 1\le k\le 2^p,
$$ 
vectors 
$$
\bu_k^{(p)}(\bt),\quad \bv_k^{(p)}(\bt)
$$ 
in \eqref{descomposicion}  depend linearly on $\bt$, and consequently by \eqref{dirdes}, so do decomposition directions 
$$
\bU_k^{(p)}(\bt),\quad \bV_k^{(p)}(\bt).
$$ 
\item For each $\bt\in\Theta_l$, the set $\bT(\bt)$ is convex. 
\end{enumerate}
\end{proposition}
\begin{proof}
For the first part, note that 
if we resort to \eqref{equis} for $u$ and $v$, we realize that for $p=m-1$, since 
vectors 
$\bu_k^{(m)}$, $\bv_k^{(m)}$ 
are given and fixed (taken, respectively, from the support of $\nabla u$ and $\nabla v$ according to \eqref{soporte}), those formulas are linear in the components of $\bt$ because weights $\lambda_k^{(m-1)}$ are as a consequence of \eqref{pesosz} and  \eqref{pesofijo}. On the other hand, for 
$$
p=m-2, m-3, \dots, 4, 3, \quad\hbox{(the levels corresponding to the dyadic processes)},
$$ 
\eqref{equis} indicate that  
$\bu_k^{(p)}$,  $\bv_k^{(p)}$
depend linearly on 
$\bu_k^{(p+1)}$, $\bv_k^{(p+1)}$
precisely because this time those relative weights $\lambda_k^{(p)}$ are exactly $1/2$ according to \eqref{constante}. For the last two levels mixing the barycenters in \eqref{baricentros} for $u$ and $v$, the linearity is ensured again because weights involved are constant. 
By the recursive nature of $(H_n)$-conditions, we have the claimed linear dependence. 

The first statement immediately yields the second. If 
$$
\br_j\in\bT(\bt),\quad j=0, 1,
$$ 
and $r\in(0, 1)$, then, for 
$$
\br=r\br_1+(1-r)\br_0,
$$
we will have
$$
\bV_k^{(p)}(\br)=r\bV_k^{(p)}(\br_1)+(1-r)\bV_k^{(p)}(\br_0),
$$
for all $k$ and $p$. 
Hence, if 
$$
\bV_k^{(p)}(\br_j)\parallel \bU_k^{(p)}(\bt),\quad j=0, 1,
$$
so will $\bV_k^{(p)}(\br)$ be. This means that $\br\in\bT(\bt)$, and $\bT(\bt)$ is convex.
\end{proof}
The non-emptiness of images is the true clue of the proof.

\section{Non-emptiness of images}\label{cinco}
Suppose $\bt\in\Theta_l$ is arbitrary, and consider the family of nested decomposition directions 
\begin{equation}\label{direccionesz}
\bU_k^{(p)}\equiv \bU_k^{(p)}(\bt),\quad 0\le p\le m-1, 1\le k\le 2^p,
\end{equation}
associated with the first component $u$ of the vector map $(u, v)\in\bbP_l\times\bbP_l$ in a bottom-to-top description. Both $\bt$ and $u$ are regarded as fixed. 

The question we would like to address is: 
\begin{quote}
Given the fixed set of decomposition directions $\{\bU_k^{(p)}\}$ in \eqref{direccionesz} coming from $u\in\bbP_l$ and $\bt\in\Theta_l$ along $\skeleton$, what are the  functions $w\in\bbP_l$ for which the probability measure corresponding to its gradient $\nabla w$ as in \eqref{gradiente},
admits a decomposition as a $(H_n)$-condition along $\skeleton$ using such same given set $\{\bU_k^{(p)}\}$ of decomposition directions  for whatever weights?
\end{quote}
We claim that the answer to this question is the full space $\bbP_l$, regardless of what the particular collection of chosen decomposition directions $\{\bU_k^{(p)}\}$, and vector of weights $\bt\in\Theta_l$ are, as long as they truly come from a function $u\in\bbP_l$, and a feasible vector of weights $\bt\in\Theta_l$. 

It should be clear how we have to resort to the top-to-bottom description of $(H_n)$-conditions to deal with this issue. 

\begin{proposition}\label{esta}
Let $\{\bU_k^{(p)}\}$ be a system of $(H_n)$-decomposition directions for $\skeleton$ coming from fixed, but arbitrary, $u\in\bbP_l$ and $\bt\in\Theta_l$. For every $w\in\bbP_l$, its corresponding probability measure $\nu_w$ in \eqref{gradiente}
admits a decomposition as a $(H_n)$-condition along $\skeleton$ through the top-to-bottom procedure, with the given system of decomposition directions $\{\bU_k^{(p)}\}$ and  some feasible vector of weights $\br\in\Theta_l$.
\end{proposition}

If this claim is correct, the non-emptiness of each set 
$$
\bT(\bt),\quad \bt\in\Theta_l, \quad u, v\in\bbP_l, 
$$ 
will hold for any arbitrary such pair of functions $u, v$.

\begin{corollary}
For every pair of functions $u, v$ in $\bbP_l$ and each $\bt\in\Theta_l$, the image 
$\bT(\bt)\subset\Theta_l$ is non-empty.
\end{corollary}
\begin{proof}
Let $u$ and $v$ in $\bbP_l$ be given. Take $\bt\in\Theta_l$. Let $\{\bU_k^{(p)}\}$ be the system of $(H_n)$-decomposition directions for $\skeleton$ corresponding to $\bt$ and $u$ as in \eqref{direccionesz}. If Proposition \ref{esta} holds, the gradient of every function $w\in\bbP_l$, in particular $\nabla v$, admits a $(H_n)$-condition representation that can be built along $\skeleton$ with decomposition directions $\{\bU_k^{(p)}\}$ and some legitimate vector of weights in $\Theta_l$. This implies our conclusion.
\end{proof}

We now prove Proposition \ref{esta}. 

\begin{proof}(of Proposition \ref{esta}). 
A first observation is that due to the upper semicontinuity claimed in Subsection \ref{cuatro}, it suffices to show the proposition for a suitable, dense subset $\clase$ of pairs $(u, \bt)\in\bbP_l\times\Theta_l$. This is further stressed below, as some of our arguments are motivated to determine such a dense class. Let, for the time being, $(u, \bt)\in\bbP_l\times\Theta_l$ be given.

We will work first separately on each set $j$ of two normals, according to our discussion in Section \ref{uno}, 
and put
\begin{equation}\label{pesos}
\bt\equiv\bt_j=(t_1^{(j)}, \dots, t_{2^m}^{(j)}),\quad m=2l+2,
\end{equation}
as in \eqref{tresjotas}. In fact, we would need to use the index $j$ to stress that ingredients to be used below could and would be distinct for different values of $j$, though the formal arguments and manipulations would be exactly the same for each $j$. To avoid such complicated notation for decomposition directions and variables, we will be dispensed with the index $j$ except for the vector of weights in \eqref{pesos}. Recall the three sub-molds $\skeletonj$ in \eqref{parcial}, and three sets of normals in \eqref{normalesz}
\begin{gather}
j=1\mapsto\hbox{normals }(1, 0), (0, 1)\mapsto\hbox{ remaining normal }(1, 1),\nonumber\\
j=2\mapsto\hbox{normals }(1, 0), (1, 1)\mapsto \hbox{ remaining normal }(0, 1),\label{normaleszz}\\ 
j=3\mapsto\hbox{normals }(0, 1), (1, 1)\mapsto \hbox{ remaining normal }(1, 0).\nonumber
\end{gather}

Take any specific value for $j=1, 2, 3$. Let
\begin{equation}\label{direcciones}
\bU_k^{(p)}, \quad 0\le p\le 2l+1,\quad 1\le k\le 2^p,
\end{equation}
be the corresponding net of decomposition directions to be used coming from some $u\in\bbP_l$ and $\bt\in\Theta_l$ along $\skeletonj$ with the top-to-bottom description. 
Let real, scalar variables
\begin{equation}\label{variables}
\bS=(S_k^{(p)})_{p=0, 1, \dots, 2l, k=1, \dots, 2^p}
\end{equation}
be used in combination with decomposition directions in \eqref{direcciones}. By this we mean that the system of vectors defined recursively by
\begin{gather}
\bu_1^{(0)}=\bbU_j\equiv\frac1{1/3}\sum_s\sum_{k \in\C_s} t_k^{(j)} \left.\nabla u\right|_{T_{s, l}},\label{bcero}\\
\bu_{2k}^{(p+1)}=\bu_k^{(p)}+S_k^{(p)}\bU_k^{(p)},\quad
\bu_{2k-1}^{(p+1)}=\bu_k^{(p)}-S_k^{(p)}\bU_k^{(p)},\label{ecuanime}
\end{gather}
for $0\le p\le 2l, 1\le k\le 2^p$,
are the mass points of the successive probability measures produced through $\skeletonj$ in a top-to-bottom form starting from $\bu_1^{(0)}=\bbU_j$ in \eqref{bcero}, and utilizing the given net of decomposition directions. 
Note how \eqref{ecuanime} implies that relative volume fractions up to level $2l$ are exactly $1/2$ as required in $\Theta_l$, and that the barycenter in \eqref{bcero} comes from \eqref{baricentros} for each $j$, and for our function $u$ as supplier, together with $\bt$, of decomposition directions.
The final collection of vectors 
\begin{equation}\label{interfases}
\cU=(\bu_k^{(2l+1)})_{1\le k\le 2^{2l+1}}
\end{equation}
of the top-to-bottom procedure will correspond in $\skeletonj$ to flat interfaces of the triangulation $\tau_l$ with either of the two normals in the set $j$ we are  currently working with. 

The main point we would like to stress is that, since decomposition vectors $\bU_k^{(p)}$ are given, constant vectors, the operation
$\bS\mapsto \cU$, given in \eqref{variables} and \eqref{interfases} through \eqref{bcero} and \eqref{ecuanime}, respectively, is linear with a range of dimension $2^{2l+1}-1$,
the number of free variables in $\bS$, most of the time. This is directly related to the dense subset  $\clase$ of $\bbP_l\times\Theta_l$ we are after. 

\begin{claim}\label{rango}
The set of pairs $(u, \bt)\in\bbP_l\times\Theta_l$ for which the range of the indicated linear operation $\bS\mapsto \cU$ has full range is dense in $\bbP_l\times\Theta_l$ (with a negligible complement). 
\end{claim}

Density in $\bbP_l\times\Theta_l$ in this statement is meant to be with respect to itself. Note that in formulas in \eqref{ecuanime}, new variables in $\bS$ are being added successively through constant vectors $\bU_k^{(p)}$. Unless there is some  dependence among these added vectors, the range of the operation $\bS\mapsto \cU$ will be maximal and equal to $2^{2l+1}-1$.

Since decomposition directions in the last level are determined by the value of $j$ (taken from the corresponding set of three normals for $\tau_l$), 
each $\bu_k^{(2l+1)}$ ought to be finally decomposed along one of those two normal directions corresponding to the value of $j$, depending on whether $\bu_k^{(2l+1)}$ corresponds to a flat interface with one or the other normal of the $j$-th set. The point we would like to stress is the following.
\begin{claim}\label{corte}
The value of the gradient of a  function $w\in\bbP_l$ in a certain element of $\tau_l$ is determined uniquely by any set of two vectors corresponding to two of the three planar interfaces of that triangle.
\end{claim}
Geometrically, this is nothing but the elementary fact that any two straight lines in the plane with directions parallel to two of our normals and passing through arbitrary, different points in $\R^2$ (corresponding to those two flat interfaces) always intersect in a single point (the value of the gradient $\nabla w$ at the given triangle). 

In this manner, the collection of mass vectors $\cU$ in \eqref{interfases} corresponding to flat interfaces of $\tau_l$ coming from the $(H_n)$-condition at the level $2l+1$, must determine, in a unique way through normals corresponding to the $j$-th set, the full set of constant gradients 
\begin{equation}\label{compatibilidad}
\left.\nabla w\right|_{T_{s, l}}=\tilde\bu_k^{(2l+1)},\quad k\in\C_s,
\end{equation}
for some $w\in\bbP_l$. If we set
\begin{equation}\label{final}
\tilde\cU=(\tilde\bu_k^{(2l+1)})_{1\le k\le 2^{2l+1}},
\end{equation}
we still have that the passage $\bS\mapsto\tilde\cU$ is a linear operation with the same dimension of its range as that of $\bS\mapsto\cU$. Refer to Figure \ref{una} for a simple, schematic draft of set $\cU$ and $\tilde\cU$. 

\begin{figure}[b]
\includegraphics[scale=0.35]{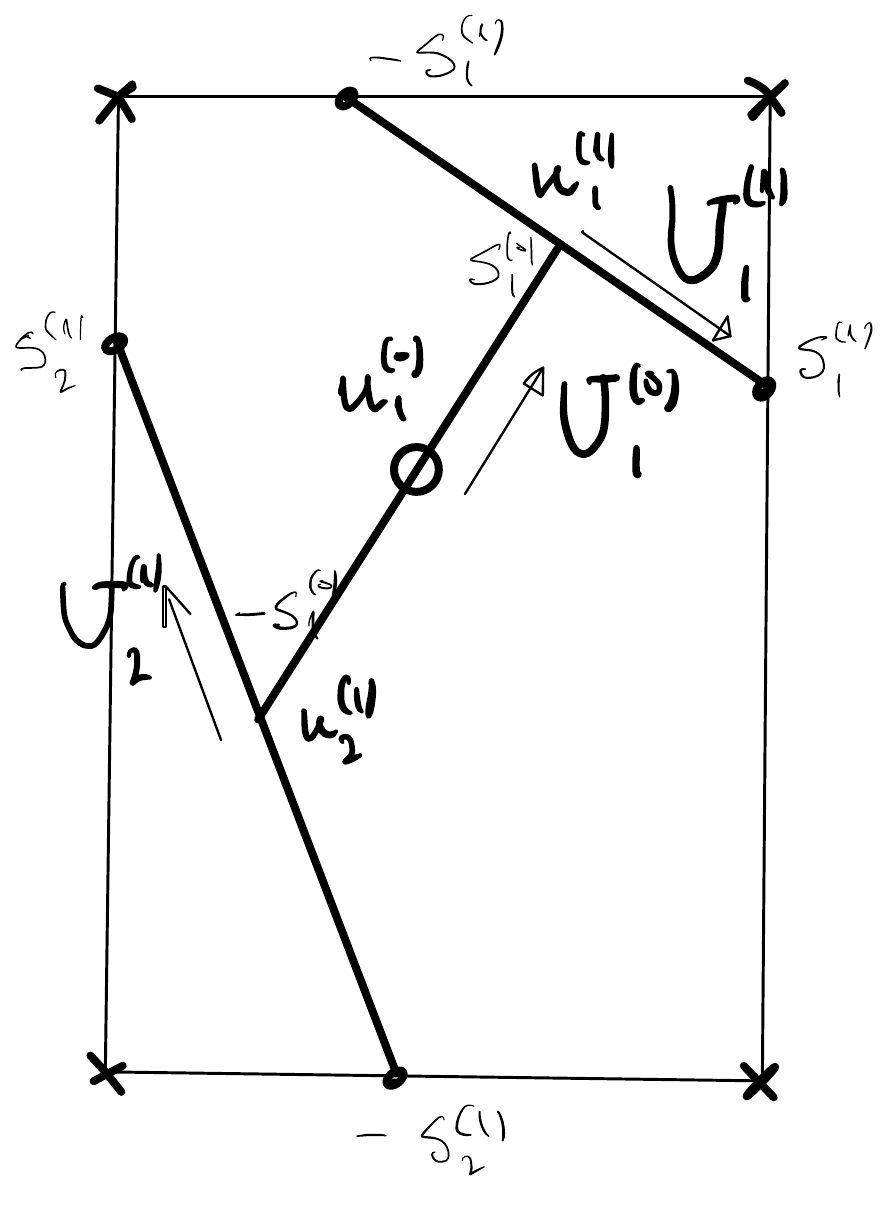}
\caption{The set $\cU$ corresponds to points with a black dot, while $\tilde\cU$ is made up of those with a cross.} \label{una}
\end{figure}

There are now three steps to be covered for the full proof.

Step 1. For \eqref{compatibilidad} to be valid for some $w\in\bbP_l$, we should enforce that those constant vectors at each element $T_{s, l}$ be compatible across interfaces with the third normal not considered in the $j$-th set of normals. Refer to \eqref{normaleszz}. 
We need to demand this explicitly as a single linear constraint on the set $\tilde\cU$ in \eqref{final} for each such planar interface with the third normal. Altogether we will have to enforce $2^{2l}$ such linear constraints. If we impose these restrictions on the set of variables $\bS$, we would end up with a dimension of the range of the linear operation $\bS\mapsto\tilde\cU$ equals to 
$$
2^{2l+1}-1-2^{2l}=2^{2l}-1.
$$
This calculation means that our procedure to generate $(H_n)$-conditions along $\skeletonj$, compatible with all planar interfaces in $\tau_l$ and with the given set of decomposition directions is endowed with $2^{2l}-1$ free parameters. We will say that values for $\bS$ which comply with these sets of $2^{2l}$ linear constraints are feasible. 

Step 2. We refer to Figure \ref{una} to better understand the following issue. Each vector $\bu_i^{(2l+1)}$ (each black dot in Figure \ref{una}) corresponding to an interface (recall that we use the index $i$ to identify planar interfaces for $\tau_l$) with one of the two normals for each value of $j$, will participate in two adjacent triangles of $\tau_l$. Therefore, we need to guarantee that the two associated intersection points (corresponding dots with a cross in Figure \ref{una}), two vectors from $\tilde\cU$ according to our discussion above and corresponding to those two adjacent elements of $\tau_l$, will occur in different sides of $\bu_i^{(2l+1)}$ along the straight line with direction given by the associated normal to interface $i$. Only when this is so would we be able to find a legitimate convex decomposition of $\bu_i^{(2l+1)}$ along the final decomposition direction, i.e. the normal corresponding to interface $i$, leading to some positive relative volume fractions $r_{2i}$, $r_{2i-1}$, on the two adjacent triangles sharing interface $i$ associated with $\bu_i^{(2l+1)}$, and leading to those two corresponding elements of $\tilde\cU$ in \eqref{final}.  We summarize this property in the following claim.

\begin{claim}\label{positivo}
Each vector $\bu_i^{(2l+1)}\in\cU$ can be decomposed as a convex combination, along its associated normal to interface $i$, of two vectors of $\tilde\cU$. 
\end{claim}

To ensure this property, we note that the dimension $2^{2l}-1$ of our linear procedure in the previous step is 
exactly the dimension of gradients in $\bbP_l$ since there are $2^{2l}$ nodes to define functions in $\bbP_l$, but gradients are determined up to an arbitrary additive constant. On the other hand, the passage from nodal values of functions $w\in\bbP_l$ to values of its gradient $\nabla w$ on elements of $\tau_l$ is a linear operation as well. We would like to use linearity together with this matching of dimensions to conclude that all gradients $\nabla w$ for every $w\in\bbP_l$ can be achieved in $\tilde\cU$ in \eqref{final} by $(H_n)$-conditions with the initially given set of decomposition directions in \eqref{direcciones}, and certain values for the variables in $\bS$ complying with those constraints across flat interfaces for the third normal not selected in the set $j$, and guaranteeing  Claim \ref{positivo}. 

As remarked earlier at the beginning of the proof, to take care of this issue we use a standard density or genericity argument by which it suffices to show our conclusion for a suitable, dense class $\clase$ of feasible pairs $(u, \bt)\in\bbP_l\times\Theta_l$
determining decomposition directions $\{\bU_k^{(p)}\}$ to be used along $\skeletonj$ to produce $(H_n)$-conditions. This dense class $\clase$ is determined by the properties:
\begin{enumerate}
\item the values $\left.\nabla u\right|_{T_{s, l}}$ of $\nabla u$ on the elements of $\tau_l$ are all different (to facilitate our discussion avoiding the situation that different elements may share the same gradient of $u$);
\item the set of decomposition direction in the full set $\{\bU_k^{(p)}\}$ allow for a full range of the map $\bS\mapsto \cU$ according to Claim \ref{rango};
\item $\bt\in\Theta_l$ actually belongs to the interior of $\Theta_l$ (with respect to itself).
\end{enumerate}
The density of $\clase$ in $\bbP_l\times\Theta_l$ is clear: if a pair $(u, \bt)\in\bbP_l\times\Theta_l$ does not belong to $\clase$, a slight perturbation of it will move it inside $\clase$. 

Let $(u, \bt)\in\clase$ be given, and determine the set of corresponding decomposition directions along $\skeletonj$ as in \eqref{direcciones}. These are given once and for all. 
Let $\pele_l$ be the subset of functions in $\bbP_l$ whose gradients can be matched in $\tilde\cU$ given in \eqref{final} and its averages over planar interfaces by $\cU$ in \eqref{interfases}, by feasible values of $\bS$ as described above, and guaranteeing, as Claim \ref{positivo} states, that legitimate convex combinations are possible at the final level passing from $\cU$ to $\tilde\cU$. We aim at showing that $\pele_l=\bbP_l$. 

On the one hand, it is elementary to realize that $u$ itself belongs to $\pele_l$ because in this case volume fractions coming from $\bt$ itself generate values for some valid $\bS_0$ through the bottom-to-top formalism. The property in Claim \ref{positivo} is trivially correct in this case. Due to the conditions imposed on $(u, \bt)\in\clase$ allowing for some flexibility in relative weights, 
one realizes that for a certain neighborhood of feasible values for $\bS$ around $\bS_0$, the full-range image of the mappings 
\begin{equation}\label{operaciones}
\bS\mapsto\tilde\cU, \quad \bS\mapsto\cU
\end{equation}
for $\bS$ in such neighborhood, will correspond to functions in some vicinity of $u$ in $\bbP_l$ (their gradients and their averages over planar interfaces, respectively). This is so because of the matching of dimensions discussed earlier. In particular, such vicinity of $u$ will also belong to $\pele_l$, and the subspace generated by $\pele_l$ will be the whole $\bbP_l$. This fact, together again with the linearity of the previous operations, and the matching of dimensions, imply that in fact the full image of both linear operations in \eqref{operaciones} for the full set of feasible values for $\bS$ is the whole space $\bbP_l$ through their gradients and averages over planar interfaces, respectively. But since for true gradients of functions in $\bbP_l$, averages over planar interfaces are alway legitimate convex combinations of the values of gradients in corresponding adjacent elements, we conclude that the full image of those linear operators in \eqref{operaciones} comply with Claim \ref{positivo}, i.e. $\pele_l=\bbP_l$. 

Finally, by density and the upper continuity claimed in Subsection \ref{cuatro}, we can deduce that indeed $\pele_l\equiv\bbP_l$ for every pair $(u, \bt)$ in $\bbP_l\times\Theta_l$. 

The full system of vectors $\tilde\cU$ in \eqref{final} will therefore determine in a unique way the values of the gradients $\nabla w$ on elements of $\tau_l$, and of weights 
\begin{equation}\label{fracinter}
\br=(r_1, r_2, \dots, r_{2l+2}),\quad r_{2i}+r_{2i-1}=2^{-1-2l}.
\end{equation}
In fact, because the way in which the set of vectors $\cU$ in \eqref{interfases} and \eqref{bcero}-\eqref{ecuanime} have been  determined, all weights at the last level are $2^{-1-2l}$ since relative volume fractions at each level have been forced to be $1/2$.

The discussion in this step can be summarized in the following statement.

\begin{quote}
Given a net of decomposition directions $\{\bU^{(p)}_k\}$ coming from some $u\in\bbP_l$ and $\bt\in\Theta_l$ given $\skeletonj$, for each gradient $\nabla w$, $w\in\bbP_l$, there are feasible values for the variables $\bS$ in \eqref{variables} and associated vector of weights in \eqref{fracinter}, depending on $j$, such that the corresponding set of vectors in $\tilde\cU$ in \eqref{final} yields precisely the values
$$
\{\left.\nabla w\right|_{T_{s, l}}\}.
$$
\end{quote}

If we recover index $j$ to distinguish the three parallel processes, 
we have thus shown that for every given $\nabla w$ for $w\in\bbP_l$, there are three vectors $\br_j$, 
\begin{gather}
\br=(\br_1, \br_2, \br_3),\label{treserres}\\
\br_1=(r_1, \dots, r_{2^{2l+2}}), \quad \br_2=(r_{2^{2l+2}+1}, \dots, r_{2\times2^{2l+2}}), \label{treserresd}\\\br_3=(r_{2\times2^{2l+2}+1}, \dots, r_{3\times2^{2l+2}});\nonumber
\end{gather}
and respective laminates $\mu_j$ along $\skeletonj$ and the given system of decomposition directions $\bU^{(p)}_{k, j}$ starting from the first moment $\bbU_j$ given in \eqref{bcero}. Because according to \eqref{firstcero},
$$
\bcero=\frac13(\bbU_1+\bbU_2+\bbU_3),
$$
 the convex combination 
\begin{equation}\label{total}
\mu=\frac13(\mu_1+\mu_2+\mu_3)
\end{equation}
is a laminate with vanishing barycenter and associated vector of weights $\br$ in \eqref{treserres}. We use the same directions and weights in the final upper levels (the ones after the independent dyadic processes for each $j$ to mix them up together) as furnished by $u$ and $\bt$ through $\skeleton$  in order to pass from the $\mu_j'$s to $\mu$ in \eqref{total}. Weights in \eqref{treserresd} provide, through classes $\C_s$, a certain full volume fraction for each triangle $T_{s, l}$ of $\tau_l$. Note that vectors $\bbU_j$, determined by $u$ in \eqref{bcero}, are the same for every $w\in\bbP_l$. 

Step 3. It might happen that the overall weight that the laminate $\mu$ in \eqref{total} assigns to each triangle $T_{s, l}$ is not $2^{-2l-1}$, as it should. We notice, however, that for each $w\in\bbP_l$ with underlying probability measure $\nu_0\equiv\nu_w$ as in \eqref{gradiente}, there are $2^{2l}$ different functions $w_q$ in $\bbP_l$, $1\le q\le 2^{2l}$, (including $w$ itself) generating the same underlying probability measure $\nu_0$: $\nu_{w_q}=\nu_0$ for all $q$. 
This is a typical effect of periodicity and translation: the vectors in the support of  $\nu_0$, the values of the gradient of $\nabla w$ in the triangles of $\tau_l$, can be moved to be the gradients of another $\tilde w\in\bbP_l$  in any other translated configuration of $\tau_l$. Hence, for each $w_q$ we will have a laminate $\mu_q$ as in \eqref{total} together with a corresponding vector $\br_q$ as in \eqref{treserres}. The convex combination
\begin{equation}\label{finalmente}
\mu=2^{-2l}\sum_q\mu_q
\end{equation}
will also be a laminate supported in the values of the gradients of $\nabla w$ in $\tau_l$ which utilizes the same family of decomposition directions provided by $u$ and $\bt$ along $\skeleton$ because this is so for every $\mu_q$, and it must be translation-invariant by construction. This means that the mass that $\mu$ in \eqref{finalmente} assigns to  triangles in $\tau_l$ that can be identified by translation with each other should be identical. 

Once we have this, the reason to conclude that the mass that $\mu$  assigns to every triangle in $\tau_l$ is precisely $2^{-2l-1}$ is elementary because of three simple facts:
\begin{enumerate}
\item the support of $\mu$ and $\nabla w$ is the same;
\item the average values $\F_1$ and $\F_2$ of both probability measures over the two classes of triangles of $\tau_l$ related by translation, as just pointed out, are the same too because both measures assign the same mass to triangles in each class;
\item they both have the null vector $\bcero$ as barycenter.
\end{enumerate}

This argument implies that the vector of weights
$$
\br=2^{-2l}\sum_q\br_q
$$ 
ought to belong to $\Theta_l$ (recall the discussion in Proposition \ref{convexidad} about the convexity of $\bT(\bt)$), and the proof is finished.
\end{proof}

\section{Conclusion}
Once the hypotheses of Theorem \ref{kakutani} have been checked out for our map $\bT$ in Definition \ref{mapeo} (for arbitrary pairs $(u, v)\in\bbP_l\times\bbP_l$, mold $\skeleton$, and arbitrary $l$) in Subsection \ref{cuatro} and Section \ref{cinco}, that theorem lets us conclude the existence of at least one fixed point of it in $\Theta_l$. Proposition \ref{facil} then leads to the desired, fundamental conclusion expressed in Theorem \ref{objetivo}: every discrete gradient probability measure supported in $\R^{2\times2}$ corresponding to a pair in $\bbP_l\times\bbP_l$, for arbitrary $l\in\N$, is a laminate. As a consequence, rank-one convexity implies quasi convexity for two-dimensional, two-component maps.


\end{document}